\documentclass[reqno,a4paper,11pt]{amsart}

\usepackage{mathtools} 
\usepackage{amsmath,amsthm,amsbsy,amssymb, comment}
\usepackage{tabu} 

\usepackage[usenames,dvipsnames,svgnames,table]{xcolor}
\usepackage[skip=5pt,font=footnotesize,labelfont=bf]{caption}
\usepackage[backref=page]{hyperref} 
\usepackage{subfigure}  
\usepackage[pdftex]{graphicx}
\graphicspath{{IMAGES/}{IMAGES1/}{./}{FACTS_POZE/}}

\usepackage{float} 
\usepackage{wrapfig}

\usepackage[full]{textcomp}
\usepackage{fourier}

\usepackage{xspace}
\usepackage{latexsym}
\usepackage{euscript}

\usepackage{tikz}
\usepackage{tikz-qtree}

\newcommand{\cL}{\mathcal{L}}

\newcommand{\cM}{\mathcal{M}}
\newcommand{\cN}{\mathcal{N}}

\newcommand{\cT}{\mathcal{T}}

\newcommand{\bx}{\mathbf{x}}

\newcommand{\FF}{\mathbb{F}}

\newcommand{\FP}{\texttt{FP}\xspace}
\newcommand{\RP}{\texttt{RP}\xspace}
\newcommand{\FPs}{\texttt{FP}s\xspace}
\newcommand{\RPs}{\texttt{RP}s\xspace}

\newtheorem{theorem}{Theorem}
\theoremstyle{plain}

\newtheorem{conjecture}{Conjecture}

 \newtheorem*{problem}{Problem}

\newcommand{\twopartdef}[4]
{
 \left\{
 	\begin{array}{ll}
	#1 & \mbox{if } #2\\
	#3 & \mbox{if } #4
	\end{array}
\right.
}

\newcommand{\rom}[1]{\uppercase\expandafter{\romannumeral #1\relax}}

\makeatletter
\let\@@pmod\pmod
\DeclareRobustCommand{\pmod}{\@ifstar\@pmods\@@pmod}
\def\@pmods#1{\mkern4mu({\operator@font mod}\mkern 6mu#1)}
\makeatother

\usepackage{MnSymbol} 

\makeatletter
\newcommand{\TriPIn}[2]{%
  \mathbin{\vphantom{\largetriangleup}\ooalign{%
    $\m@th#1\largetriangleup$\cr
    \hidewidth$\m@th#1$\scriptsize\raisebox{0.5pt}{$\times$}\hidewidth\cr}}%
}
\makeatother

\makeatletter
\newcommand{\TriSIn}[2]{%
  \mathbin{\vphantom{\largetriangleup}\ooalign{%
    $\m@th#1\largetriangleup$\cr
    \hidewidth$\m@th#1$\scriptsize\raisebox{1pt}{$+$}\hidewidth\cr}}%
}
\makeatother

\begin{document}



\title{Factorials $\pmod* p$ and the average of modular mappings}
\author{Cristian Cobeli,  Alexandru Zaharescu}

\address{
CC: \textit{Simion Stoilow} Institute of Mathematics of the Romanian Academy, 
21 Calea Grivi\c tei Street,
P. O. Box 1-764, RO-014700, Bucharest, Romania}
\email{cristian.cobeli@imar.ro}

\address{
AZ: Department of Mathematics, 
University of Illinois at Urbana-Champaign,
Altgeld Hall, 1409 W. Green Street,
Urbana, IL, 61801, USA 
and 
\textit{Simion Stoilow} Institute of Mathematics of the Romanian Academy, 
21 Calea Grivi\c tei Street,
P. O. Box 1-764, RO-014700, Bucharest, Romania}
\email{zaharesc@illinois.edu}

\begin{abstract}
We have known that most sequences in $\mathcal{M}=\{1,2,\dots, M\}$ with length $n$ 
will miss $Me^{-\lambda}$ of the total numbers of  $\{1,2,\dots,M\}$ as the ratio
$n/M$ tends to $\lambda$. Now we consider a more general case where the numbers in
$\{1,2,\dots,M\}$ are achieved exactly k times by a 'random' sequence $f(1), f(2),\dots,f(n)$. We 
show that if $n/M\rightarrow \lambda$, then the limit has a Poisson distribution, that is, the 
proportion of sequences for which some number in $\mathcal{M}$ is achieved  exactly $k$ times has 
the limit $\frac{\lambda^k}{k!}e^{-\lambda}$.
We conjecture that this is the behavior of the factorial mapping modulo a prime and 
present a few supporting arguments.
\end{abstract}
\maketitle

 \section{Introduction}


\begin{figure}[t]
\centering
 \hspace*{1mm}  
     \includegraphics[width=0.98\textwidth]{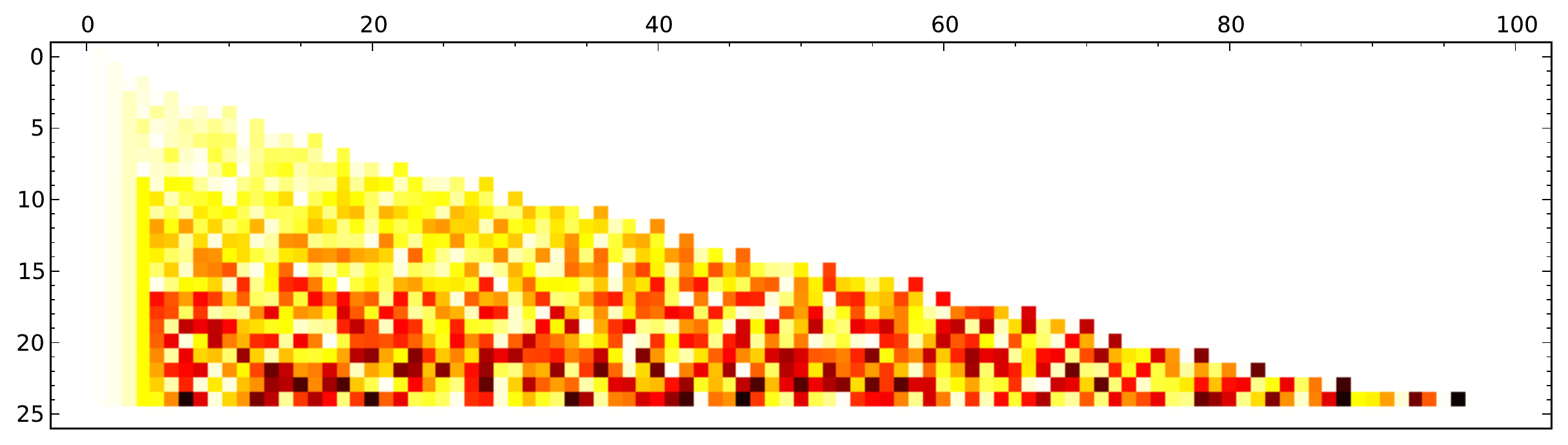}\\
     \includegraphics[width=0.99\textwidth]{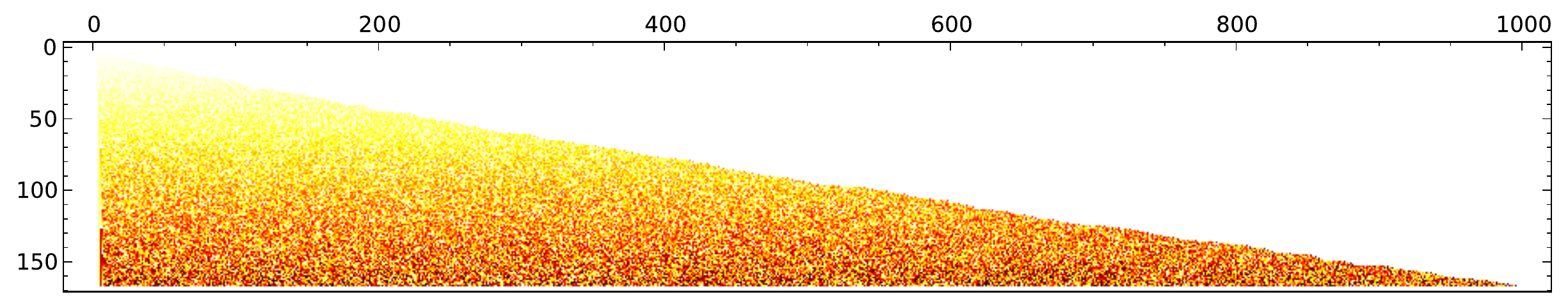}
\caption{The residue classes $n! \pmod* p$, $1\le n\le p-1$, for primes $p <100$, top and for
primes $p< 1000$, bottom. Different colors correspond to different residue classes.
}
 \label{Figure1}
 \end{figure}

\noindent
In contrast with many of their real siblings, which are continuous, modular functions are known to
be chaotic in nature. This is true for polynomials, monomials in particular, and this behavior
accentuates as the power increases from $2$ to the value of the modulus minus $2$, that is, from 
squares to inverses.  The matter was studied extensively in 
the works of 
the authors et al. 
~\cite{CZ2000, CZ2001, CVZ2000a, CVZ2003, CGZ2003},
P. Kurlberg et al.~\cite{KR1999, Kur2000, Kur2009},
and
I. E. Shparlinski~\cite{Shp2012}.

Since the factorial mapping $x\mapsto x!$ modulo a prime $p$ acts like a 'diagonal of the
monomials', one expects it to induce an even higher degree of randomness. Many visual
representations, such as those in Figure~\ref{Figure1}, confirm the expectation.
One of the main issues is raised by the following problem.
\begin{conjecture}[R. Stauduhar]\label{ConjectureStauduhar}
     Let $p$ be a prime and let $h(p)$ be the number of distinct residues of 
$1!, 2!,\dots,(p-1)!\pmod*  p$. Then 
\[
\lim_{p\to \infty}\frac{h(p)}{p}=1-\frac 1e\,.     
\]
\end{conjecture}

Stauduhar's Conjecture is stated as Problem~$7$ in a list of $114$ problems, which, according to a
nice tradition of the time, appeared  in the Proceedings of
the 1963 Number Theory Conference~\cite{NTC1963}.  
In several later articles, various authors referred to the same problem, but in the formulation of
R. K. Guy~\cite[Problem F11]{Guy2004}. 
There,  Guy draws attention to the phenomenon that the sequence of factorials 
$1!, 2!,\dots,p!$ misses about $p/e$ residue classes modulo $p$. 
The conjecture remains unsolved, but it was proved~\cite{CVZ2000b} that it holds in
average for all modular mappings or, in other words, 
this is a characteristic of randomly chosen sequences of $p$ classes of residues modulo $p$.
The link with the 'randomness' was made earlier, as we have found in the last stage of the
preparation of this manuscript. Thus, Brillhart~\cite[Problem 7, page 90]{NTC1963}  
adds a comment to the statement of the conjecture, saying in parentheses  that
\textit{``from extensive numerical calculation
the statement appears to be true''} and \textit{'this is the ``random result'' for any set of 
numbers'}.


In this article we show that in average, the expected number of residue classes 
that are reached  exactly $k\ge 0$ times by a random modular function defined from 
$\FF_p$ with values in $\FF_p$ equals $\frac{1}{k!e}$. The limit asymptotic values for the mean 
and variation are obtained in Theorems~\ref{Theorem1} and the size of the error term are evaluated 
in \ref{Theorem3} below. 
Accordingly, for factorials, this allows us to extend Stauduhar's Conjecture, since 
wide-ranging numerical verifications confirm this trend.

\begin{conjecture}\label{ConjecturePoisson}
     For any integer $k\ge 0$, the proportion of elements $y\in\FF_p$ for which there are 
exactly $k$ positive integers $n$ for which $n!\equiv y\pmod* p$ tends to the limit 
$\frac{1}{k!}e^{-1}$, as $p\to\infty$.
\end{conjecture}
Notice that Stauduhar's Conjecture is just the $k=0$ case of Conjecture~\ref{ConjecturePoisson} and 
the general statement says that if $p$ is sufficiently large, then the elements of the sequence 
$\{n! \pmod* p\}_{1\le n\le p}$ behave like a Poisson process with mean $\lambda=1$.
More background data are presented in \mbox{Section~\ref{SectionRamificationPoints}.}

Our main result is a uniform estimation over all modular mappings 
of the average number of elements attained exactly a certain fixed number of times.
Let $\cM$ be a finite set of $M$ elements and let $\cN\subset\cM$ be a subset of $n$ elements. 
Denote by $\cT$ the set of $n$-tuples with components in $\cM$, that is, $\cT=\cM^n$ or
$\cT=\big\{ \bx =(x_1,x_2,\dots,x_n) :\  \  x_1,\dots,x_n \in \cM \big\}$.
We say that an element \textit{$y \in \cN$ is represented $k$ times in the tuple $\bx\in\cT$}, if
exactly $k$ components of $\bx$ coincide \mbox{with $y$.} 
Let $m_k(\bx)$ be the
number of elements of $\;\cN$ that are represented exactly $k$ times by $\bx$, \mbox{that is},
\begin{equation}\label{eqmkx}
	\begin{split}
	  m_k(\bx):=\Big|\big\{y\in \cN : \ \text{ $y$ is represented exactly $k$ times in $\bx$}
	  \big\}\Big|\,.
	\end{split}
\end{equation}
It turns out that if the sets $\cN$ and $\cM$ have sufficiently many elements, than the proportions
$m_k(\bx)/n$ cluster around the values of a Poisson distribution. A controlled bound of the 
maximum deviation is obtain in the following  theorem.
\begin{theorem}\label{TheoremMain}
Let $\lambda\in(0,1]$, $\gamma\in[0,1)$, $\delta\in(0,(1-\gamma)/2)$ and let $k\ge 0$  be integer.
Suppose the integer variables $M$ and $n$  satisfy the inequalities $k\le n\le M$  and 
$M=n/\lambda +O(n^\gamma)$, uniformly on $\lambda$ and $k$ as $n$ tends to infinity while $\gamma$ 
is fixed. Then
\begin{equation*}
  \frac{1}{|\cT|} \cdot
  \Big|\Big\{\bx :\ \Big|\frac {m_k(\bx )}{n}-\frac{\lambda^k}{k!}e^{-\lambda}\Big|<
		    n^{-\delta}\Big\}\Big|
	  = 1-O\left(\frac{\lambda^k}{k!}\cdot\frac{1}{n^{1-2\delta}}
	\Big(1+\frac{\lambda^{k+1}(\lambda+k)}{k!}n^{\gamma}\Big)\right),
\end{equation*}
and the constant involved in the big Oh term does not dependent on $\lambda, \gamma, \delta$ and 
$k$.
\end{theorem}
\noindent
Let us notice that both summands in the error term have their particular type of 
contribution. The distinction can be seen whether or not $\lambda$ is very small, or whether 
$k = 0$, or else, whether $k$ is small or it becomes large enough to make one or the other of the 
terms of the sum dominate on one side or the other of the balance point, which is attained if 
$n\approx\left(k!/(\lambda^{k+1}(\lambda+k)\right)^{1/\gamma}$.

 \section{Ramification points and factorials $\pmod* p$}\label{SectionRamificationPoints}

\noindent 
Referring to the distribution of the sequence of factorials modulo a prime $p$, an earlier simpler
problem was proposed by P. Erd\H os~\cite[problem F11]{Guy2004}. He asks whether there
exists $p>5$ for which the numbers $2!, 3!,\dots, (p-1)!$ are all distinct modulo $p$. 
If there were such a prime, then \mbox{T. Trudgian}~\cite{Tru2014}
verified that it must be greater than $10^9$. The problem of Erd\H os is still unsolved,
although B. Rokowska, A. Schinzel~\cite{RS1960}
and \mbox{Trudgian}~\cite{Tru2014} showed that for large classes of primes the sequence of
factorials modulo $p$ can not be so close to a permutation, while Klurman and Munsch~\cite{KM2017}
obtained non-trivial bounds for the average deviation.

In reality, for any given $p$, one can check that there are many residue classes
\mbox{modulo $p$} that
are hit more than once by the sequence of factorials. 
Iterating the  factorial function \mbox{modulo $p$},  two types of points
distinguish in the complex created tree.
They are the fixed points, or the roots of the tree, and the ramification points, which
are the residue classes reached several times by the factorials.

In precise terms, we say that $1\le x \le p-1$ is a \textit{fixed point} or, shortly, $x$ is an \FP
of the factorial function modulo $p$, if $x!\equiv x \pmod* p $.
The residue classes modulo $p$, viewed as stacks that are reached a different number of times by the
sequence  $1!,2!, 3!,\dots, (p-1)! \pmod* p$, are also called \textit{ramification points} of the
factorial function modulo $p$ or,
shortly, \RPs.     
Thus, rigorously, we say that  $y\in\{1,\dots, p-1\}$ is a $k$-\RP, if  it is hit exactly $k$ times
by the sequence $1!,2!, 3!,\dots, (p-1)!$, that is, there exist exactly $k$ distinct integers 
$x_1,\dots, x_k\in \{1,2,\dots,p-1\}$ such that $x_1!\equiv \cdots\equiv x_k!\equiv y\pmod* p$.
For example, in \mbox{Table}~\ref{Table23} one sees that for $p=23$, the points 
$a=1, 2, 5, 9, 12, 22$ are \FPs; 
$10, 15, 16, 17, 19, 20$ are $0$-\RPs (they are missed by the factorial function modulo $23$);
$2, 3, 4,\dots$, $9$, $11$, $12$, $13$, $14$, $18$, $21$ are $1$-\RPs (they are hit exactly once by
the factorial function);
$22$ is a $3$-\RP (it~is attained three times) and
$1$ is a $5$-\RP (it appears five times on the second row of the table).
 
\medskip
\begin{center}
\captionof{table}{The values of the modular factorial function $x\mapsto x!\pmod p$, for
$p=23$.}\label{Table23}
\setlength{\tabcolsep}{4.35pt}
\small\footnotesize
\begin{tabular}{ccccccccccccccccccccccc}
\hline
$x$
& $1$ & $2$& $3$ & $4$ & $5$& $6$ & $7$ & $8$ & $9$ & $10$ & $11$
& $12$& $13$ & $14$ & $15$& $16$ & $17$ & $18$ & $19$ & $20$& $21$ & $22$\\
$x!\pmod p$ 
& $1$ & $2$& $6$ & $1$ & $5$& $7$ & $3$ & $1$ & $9$ & $21$ & $1$
& $12$& $18$ & $22$ & $8$& $13$ & $14$ & $22$ & $4$ & $11$& $1$ & $22$\\
\hline
\end{tabular} 
\end{center}  
\medskip\medskip

\noindent
Remark that $1, 2$ and $p-1$ are always \FPs (by Wilson's Theorem), but congruences such as
\begin{equation*}\label{eqExamples}
	\begin{split}	
	  4244208!\equiv  4244208\pmod* {9991769}, 
\quad\qquad
	  5112195!\equiv 5112195\pmod* {5444407} 
	\end{split}
\end{equation*}
are sporadic. Quite often, the median $m=(p+1)/2$ is an almost trivial fixed point, too, since
$m!\equiv m\pmod* p$ or $m!\equiv m-1 \pmod* p$ for any prime $p\equiv 3 \pmod* 4$.

\smallskip

Also, notice this intrinsic connection between \FPs and the particular ramification point $1$: \\
for any prime  $p\ge 3$,
the number of \FPs is $r$   if and only if $1$ is an $(r-1)$-\RP.

Let $m(p)$  be the number of $0$-\RPs, that is, the number of residual classes
\mbox{$1, 2, \dots,p-1$} that are missed by the sequence of factorials modulo $p$.
Table~\ref{Table1} shows evidences in favor of Stauduhar's Conjecture~\ref{ConjectureStauduhar}.

\medskip
With so many residue classes missed by the factorials, we expect that the sequence of factorials
is always very far from a permutation, but a proof of this fact is still needed.
A comparison approach is taken by Lev~\cite{Lev2006} who tests the size of the partial sums of 
the elements of permutations in abelian groups.
Upper and lower bounds for the number of distinct residue classes of $n!\pmod* p$ under different 
conditions were obtained by Banks et al.~\cite{BLSS2005},
Garaev et al.~\cite{GLS2004, GH2017},
Klurman and Munsch~\cite{KM2017}. 
The interesting properties of $n!\pmod* p$ have been studied from different perspectives by 
us~\cite{CVZ2000b},
Shub and  Smale~\cite{RS1995},
Markstr\"om~\cite{Mar2015},
Luca et al.~\cite{LS2003, LS2005a, LS2005b},
Cheng~\cite{Che2004},
Broughan et al.~\cite{BB2009},
Garaev et al.
~\cite{GLS2004, GLS2005},
Banks et al.
~\cite{BLSS2005},
Dai et al.~\cite{CD2006, Dai2008},
\mbox{Garc\'ia}~\cite{Gar2007}, \cite{Gar2008}.

\smallskip
\begin{center}
\captionof{table}{The proportion of residue classes missed by factorials for different primes.
Compare the values from the third column  with $1/e\approx 0.3678794$.}\label{Table1}
\small\footnotesize
\begin{tabular}{llcc}
\hline
$n\qquad\qquad$ & $p_n$ & $\quad m(p_n)/p_n\quad$ & $1/e-m(p_n)/p_n$ \\
\hline
$5$ & $11$ & $0.5454545$ & $-0.1775751$ \\
$7$ & $17$ & $0.3529412$ & $\phantom{-}0.0149382$ \\
$8$ & $19$ & $0.4210526$ & $-0.0531731$ \\
$26$ & $101$ & $0.3663366$ & $\phantom{-}0.0015428$ \\
$100$ & $541$ & $0.3789279$ & $-0.0110484$ \\
$1000$ & $7919$ & $0.3725218$ & $-0.0046423$ \\
$10000$ & $104729$ & $0.3681884$ & $-0.0003089$ \\
$100000$ & $1299709$ & $0.3679662$ & $-0.0000867$ \\
$1000000$ & $15485863$ & $0.3676930$ & $\phantom{-}0.0001864$ \\
\hline
\end{tabular} 
\end{center}

\medskip

Together with our colleague M. V\^aj\^aitu~\cite{CVZ2000b}, we showed that when $p$ is sufficiently 
large the
number
of $0$-\RPs is about $p/e$ for almost all sequences. 
Since most sequences have no particular simple 
defining rule, we may say that the non-representation of $p/e$ residue classes modulo $p$ is a
general feature that characterizes randomness.

\begin{center}
\begin{figure}[htb]
\centering
\hfill
   \includegraphics[angle=-90,width=0.47\textwidth]{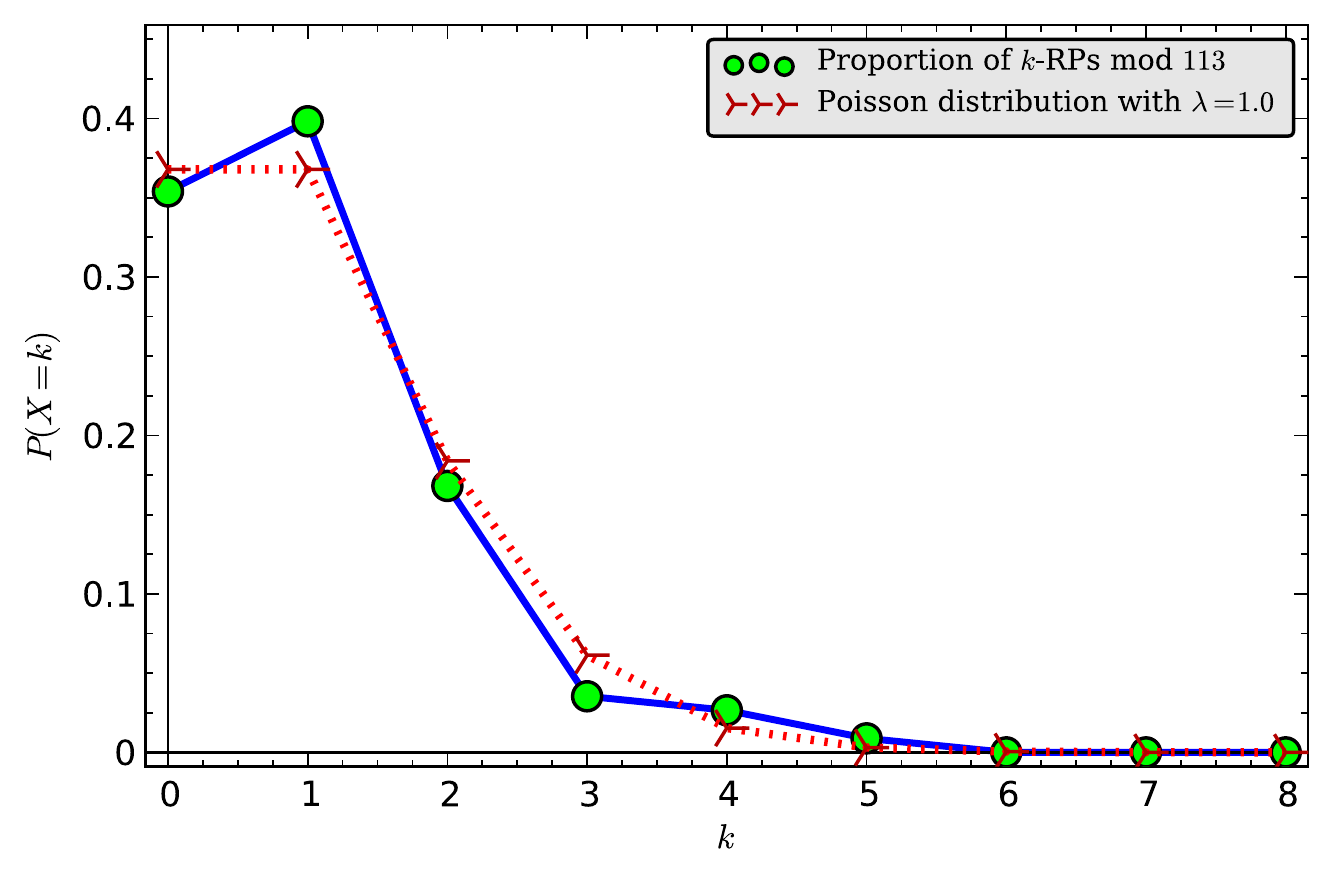}
\hfill\hfill
   \includegraphics[angle=-90,width=0.47\textwidth]{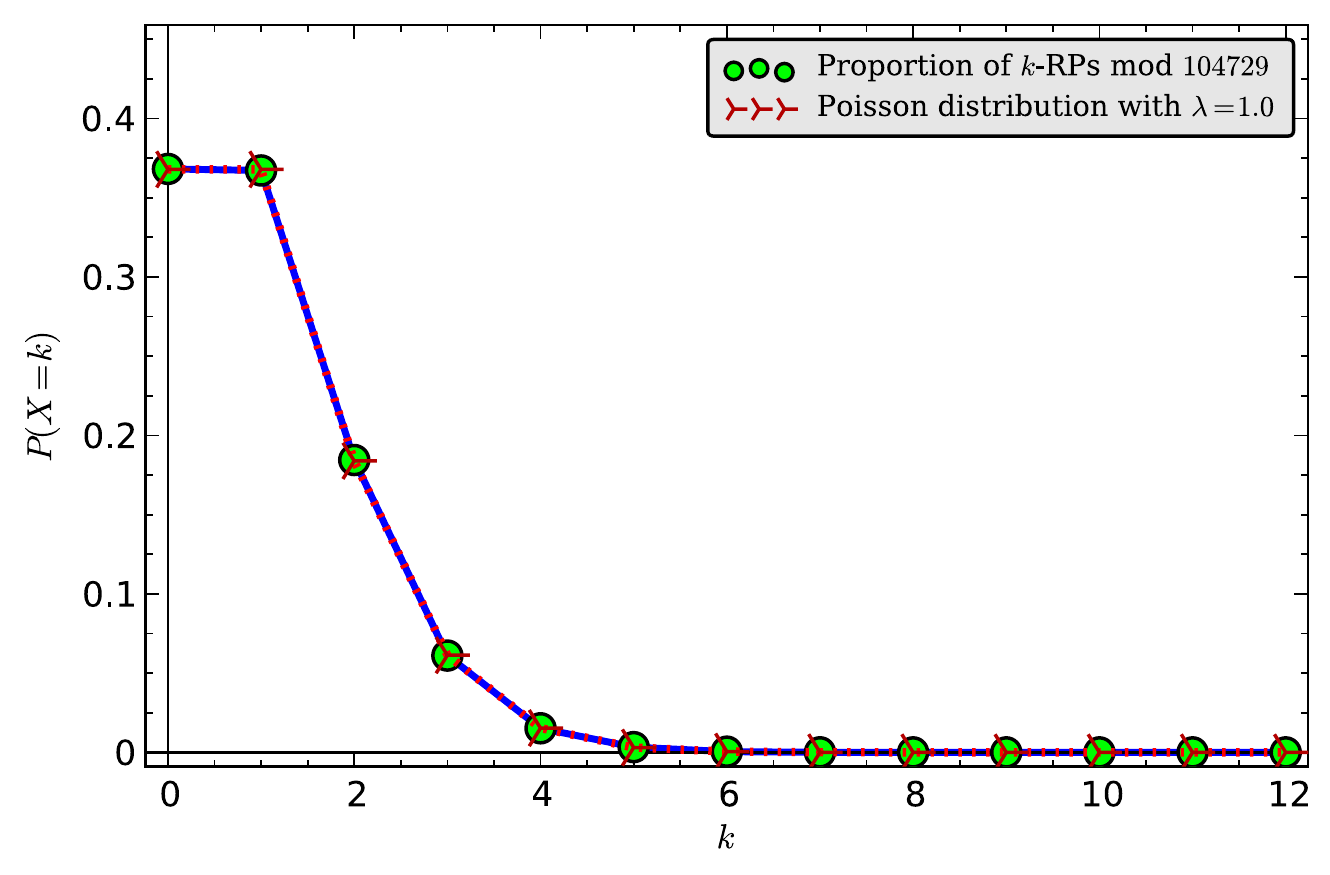}
\hfill\mbox{}
\hfill
   \includegraphics[angle=-90,width=0.47\textwidth]{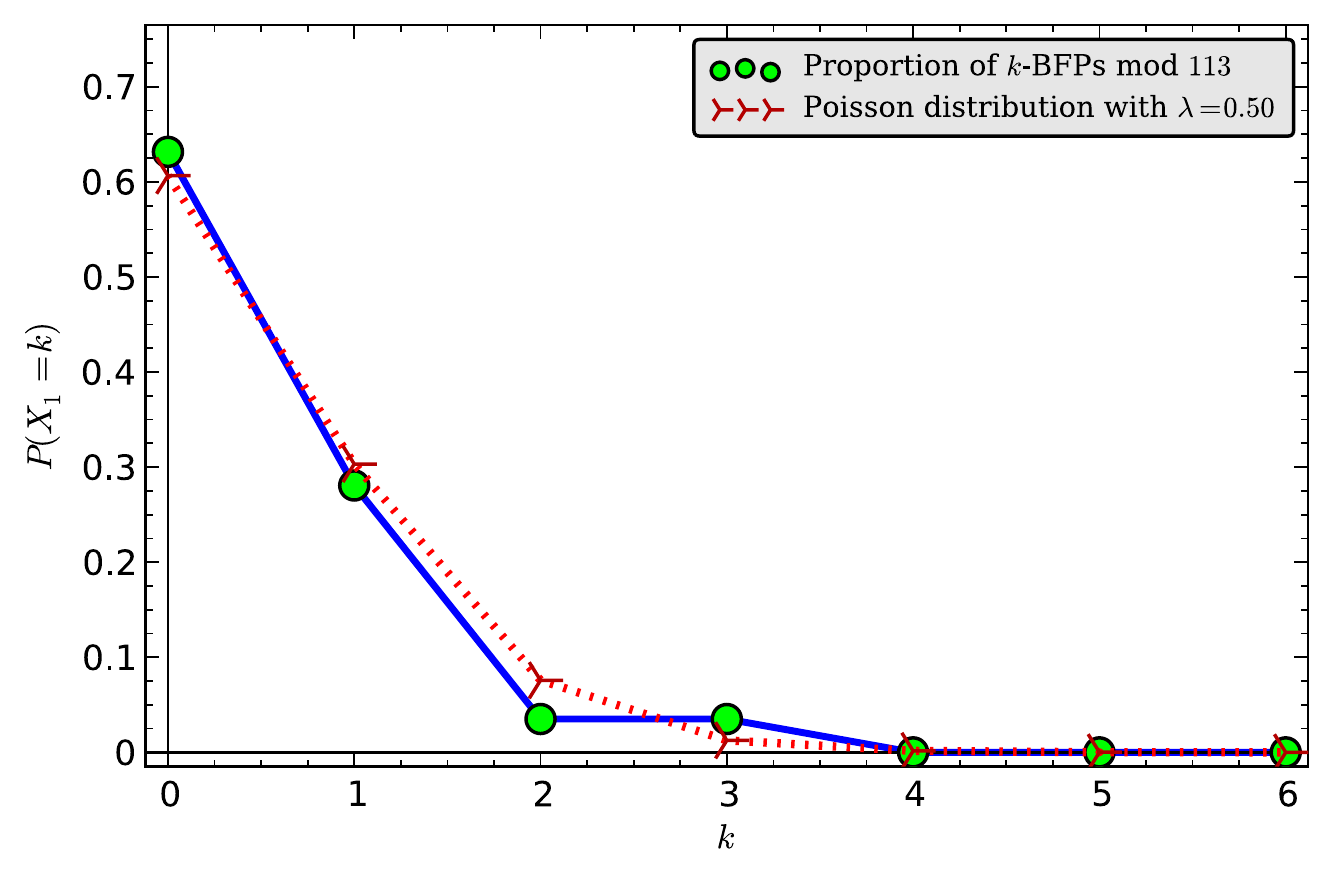}
\hfill\hfill
   \includegraphics[angle=-90,width=0.47\textwidth]{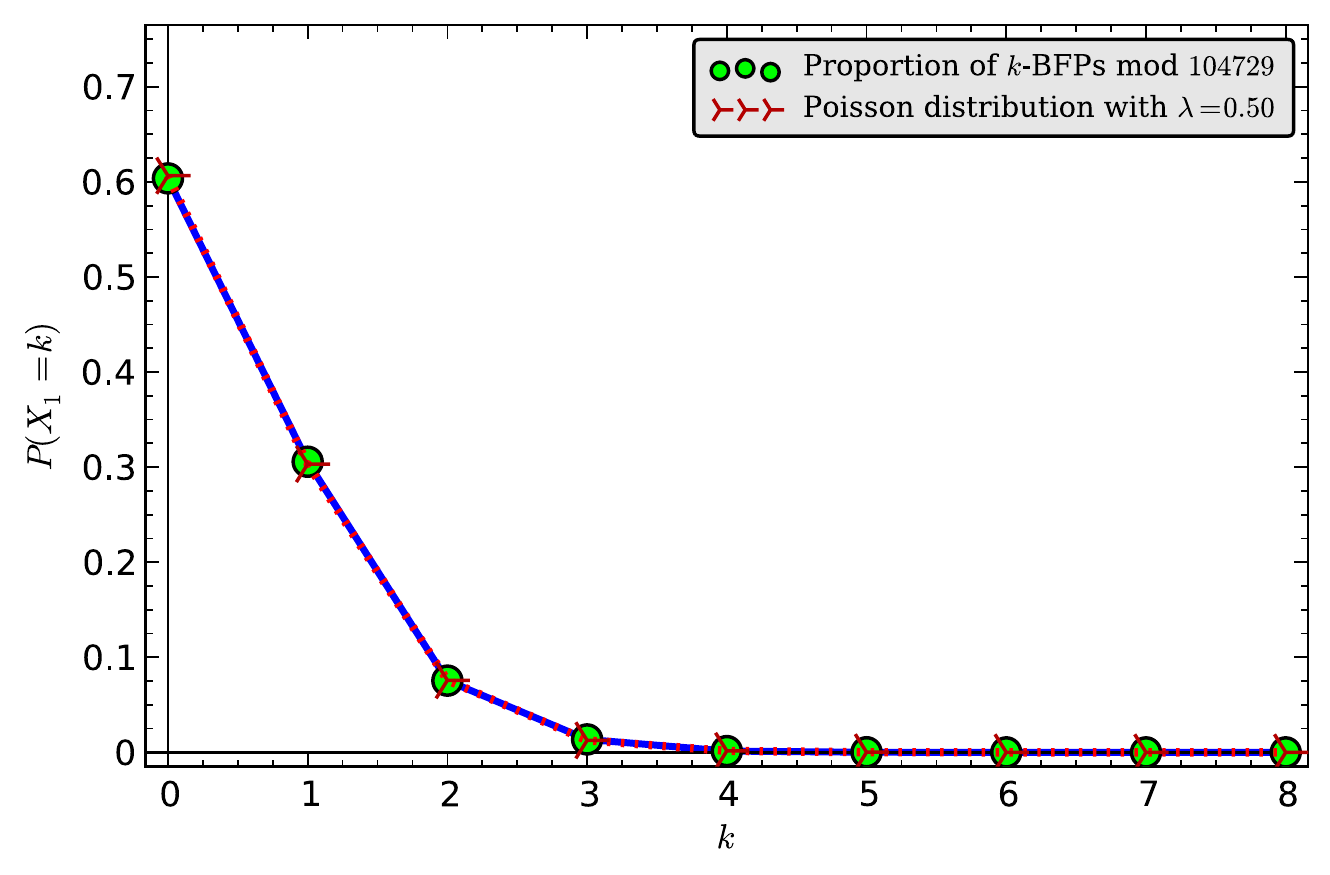}
\hfill\mbox{}
\caption{The proportions of $k$-\RPs of the $30$th prime (the figures on the left) and of the 
$10\,000$th
prime (the figures on the right). In the figures at the bottom, only the \RPs obtained from 
factorials from the first half, that
is, only the frequencies of $1!$, $2!$, $\dots$, $\big(\frac{p-1}{2}\big)! \pmod* p$ are counted. 
The proportions are compared  with the Poisson distribution with means $\lambda=1$ and
$\lambda=1/2$, \mbox{respectively}.
}
 \label{Figure2}
 \end{figure}
\end{center}

The more general statement~\cite[Theorem 1]{CVZ2000b}
says that if $\lambda \in (0,1]$ is fixed, $p$ is a large prime number, and $n\sim \lambda p$,
then almost all sequences of length $n$ chosen from  a 
subset $\cN\subset \FF_p$ having $n$ elements omit about $n/e^{\lambda}$ classes of
$\cN$.

We remark that the proportions of $k$-\RPs of the factorials are close to their conjectured limits 
even for small
primes $p$, as can be seen from the two example shown in  Figure~\ref{Figure2}
for $p=113$ and $p=104729$ (the $30$th and the $10\,000$th primes).
For $p=113$, there are $40$ points that are missed by factorials, $45$
points that are hit exactly once, $19$ points that are hit exactly twice and so on.
The record values for the two primes are  the residue classes $57$, which is 
hit exactly $5$ times, since
\begin{equation*}
	     20!\equiv  40!\equiv 50!\equiv 89!\equiv 101! \equiv 57 \pmod* {113}  \,,
\end{equation*}
and 
$78919$, which is hit exactly $9$ times, since
\begin{equation*}
\begin{split}
     2470!&\equiv 2742!\equiv 20986!\equiv 29734!\equiv 36188!\\
&\equiv 39370! \equiv 39865!\equiv40787!\equiv 65457! \equiv 78919 \pmod* {104729}\,.
\end{split}
\end{equation*}

The fast decay of the number of $k$-\RPs as $k$ increases is a 
 widespread phenomenon in different contexts and it can be investigated under 
 the following generic query.
\begin{problem}\label{Problem1}
Let $\cM$ be a set of positive 
integers and suppose $\cN\subset\cM$ is a
'large enough' subset. 
What is the expected  proportion of numbers that are represented
exactly $k$ times by the sequence $f(1)$,$f(2)$, $\dots$, $f(n)$, where
$\cN:=\big\{f(1),f(2),\dots,f(n)\big\}$?     
\end{problem}
The most expected answer to the question in this problem 
is given in the following section.

 \section{The limits of the average and of the variation}\label{SectionAverageAndMean}

\noindent 
Denote by $A(k;M,n)$ the average of  
the proportions of elements of $\cN\subset\cM$ that are represented $k$ times by vectors 
$\bx \in \cT$, that is,
\begin{equation}\label{eqrho}
     A(k;M,n):=\frac{1}{|\cT|} \sum_{\bx \in \cT}\frac {m_k(\bx) }{n}\,,
\end{equation}
where the counter $m_k(\bx)$ is defined by~\eqref{eqmkx}.
For any pair of elements $x, y \in \cM$, we define
\begin{equation*}
     \delta(x,y)=\twopartdef{1,}{x=y}{0,}{x \neq y}\,.
\end{equation*}
Similarly, for any vector $\bx=(x_1,x_2,\dots,x_n)\in \cT$ and any $y\in \cM$, we denote
\begin{equation*}
\delta_k(\bx,y)=
\begin{cases}
1,       & \text{if exactly $k$ components of $\bx$ coincide with $y$,}\\
0,  &  \text{else.}
\end{cases}
\end{equation*}
Then,  by the inclusion-exclusion principle, the counting function $m_k(\bx)$ can be expressed as
 \begin{equation}\label{eqm1}
   \begin{split}
     m_k(\bx)&=\Big|\{y\in \cN  :\  \delta_k(\bx,y)=1\}\Big|\\
&=\sum_ {y\in \cN}\sum_{\substack{ \cL\subset \{1,2,\dots, n\}\\  |\cL|=k}}  
     \prod_{\substack{ j\in \cL\\i \in \{1,2,\dots,n\}\setminus \cL} }
     \delta(x_j,y)\Big(1-\delta(x_i,y)\Big)\,.
   \end{split}
\end{equation}
Since for each $y\in \cN$ there are ${n \choose k}$ choices for the positions of the components of
$\bx$ occupied by $y$, while the remaining $M-k$ positions can take any of the remaining $M-1$
values, by \eqref{eqm1}  we obtain the following closed form
expression of the average introduced by \eqref{eqrho}:
 \begin{equation}\label{eqpk}
   \begin{split}
     A(k;M,n)&=\frac{1}{|\cT|} \sum_{\bx \in \cT}\frac {m_k(\bx)}{n}\\
     &=\frac{1}{nM^n}\sum_{\bx \in \cT}\sum_ {y\in \cN}
	  \sum_{\substack{\cL\subset \{1,2,\dots,n\}\\  |\cL|=k}}  
	  \prod_{\substack{ j\in \cL\\i \in \{1,2,\dots,n\}\setminus \cL	  }  }
	  \delta(x_j,y)\Big(1-\delta(x_i,y)\Big)\\
     &=\frac{1}{M^n}{n \choose k}\big(M-1\big)^{n-k}\,.
   \end{split}
\end{equation}
This implies that 
\begin{equation}\label{eqLimitMean}
	\begin{split}
	A(k;M,n) =\frac{1}{k!}\cdot\frac{n(n-1)\cdots(n-k+1)}{M^k}
	       \left(1-\frac{1}{M}\right)^{n-k}
		    \to \frac{\lambda^k}{k!}e^{-\lambda},
	\end{split}
\end{equation}
provided that $k$ is fixed and both $M,n\to \infty$ such that the ratio 
$n/M$ tends to $\lambda$.

\medskip

Next we find the limit of the variance of $A(k;M,n)$ about its limit mean determined 
by the assymptotic estimate~\eqref{eqLimitMean} 
Let  
\begin{equation}\label{eqM2}
     M_2(k;M,n):=\frac{1}{|\cT|}
     \sum_{\bx \in \cT}
	  \left(\frac{m_k(\bx)}{n}-\frac{\lambda^k}{k!}e^{-\lambda}\right)^2
\end{equation}
 be the second  moment about the mean. 
%
Expanding the binomial, we find that it can be \mbox{written as}
\begin{equation}\label{eqM2S2}
	\begin{split}
     M_2(k;M,n)
   &=\frac{\lambda^{2k}}{k!^2}e^{-2\lambda}-
     \frac{2\lambda^k}{k!}e^{-\lambda}A(k;M,n) +S_2(k;M,n)\,,
	\end{split}
\end{equation}
where we denoted
\begin{equation*}
S_2(k;M,n)=\frac{1}{|\cT|}
	  \sum_{\bx \in \cT}\left(\frac{m_k(\bx)}{n}\right)^2
     =\frac{1}{M^nn^2}\sum_{\bx \in \cT}m^2_k(\bx).
\end{equation*}
By \eqref{eqm1} and the same argument used to derive the expression \eqref{eqpk},
we see that
\begin{equation}\label{eqS2bis}
\begin{split}
     S_2(k;M,n)&=\frac{1}{M^nn^2}\sum_{\bx \in \cT}
			 \Bigg(\sum_ {y\in \cN}\delta_k(\bx,y)\Bigg)^2\\
&=\frac{1}{M^nn^2}\sum_{\bx\in \cT}
     \sum_ {y,y' \in \cN}\delta_k(\bx,y)\delta_k(\bx,y')\\
&=\frac{1}{M^nn^2}\Bigg(\sum_{\bx\in \cT}
     \sum_ {y\not = y' \in \cN}\delta_k(\bx,y)\delta_k(\bx,y')
     +\sum_{\bx\in \cT}
     \sum_ {y=y' \in\cN}\delta_k(\bx,y)\delta_k(\bx,y')\Bigg)\\
&=\frac{1}{M^nn^2}\left(n(n-1){n \choose k}{n-k \choose k}(M-2)^{n-2k}
     +n{n \choose k}(M-1)^{n-k}\right).
\end{split}
\end{equation}
We denote and rewrite the two terms on the last line of \eqref{eqS2bis} as
\begin{equation}\label{eqSISII}
	\begin{split}
	     S_{I} (k;M,n)=&
\left(1-\frac{1}{n}\right)\cdot \frac{1}{k!^2}\cdot\frac{n(n-1)\cdots (n-k+1)}{M^k}\\
\phantom{S_{I} (k;M,n)=}&\phantom{\left(1-\frac{1}{n}\right)\cdot \frac{1}{k!^2}}\times
	\frac{(n-k)(n-k-1)\cdots (n-2k+1)}{M^k}\cdot\left(1-\frac{2}{M}\right)^{n-2k},\\
	  S_{II} (k;M,n)=&
     \frac{1}{n}\cdot \frac{1}{k!}\cdot\frac{n(n-1)\cdots (n-k+1)}{M^k}\cdot
	\left(1-\frac{1}{M}\right)^{n-k}.
	\end{split}
\end{equation}
Provided that $k$ is fixed and both $M, n\to \infty$ such that the ratio 
$n/M$ tends to some fixed constant $\lambda\in (0,1]$, it follows that
\begin{equation*}
	\begin{split}
	S_{I} (k;M,n) \to \frac{\lambda^{2k}}{k!^2}e^{-2\lambda}\quad \text{ and }\quad
	   S_{II} (k;M,n) \to 0\,,
	\end{split}
\end{equation*}
which means that
\begin{equation*}
	\begin{split}
	\lim_{\begin{subarray}{c}n,M\rightarrow \infty\\ \frac{n}{M}\rightarrow \lambda
\end{subarray}}S_2(k;M,n)=\left(\frac{\lambda^k}{k!}e^{-\lambda}\right)^2.
	\end{split}
\end{equation*}
On inserting this limit and the asymptotic estimate~\eqref{eqLimitMean} on the right side of 
relation~\eqref{eqM2S2},  we find that the limit of the second moment about the mean is zero.

\begin{theorem}\label{Theorem1}
Let $\lambda \in (0,1]$ and let $k$ be a fixed integer. 
Suppose that $n$ and $M$ are integer variables and both increase tending to infinity while their 
ratio $n/M$ tends to $\lambda$.
Then, the average defined by relation~\eqref{eqrho} and the second square moment defined 
by~\eqref{eqM2} have the following limits:
\begin{equation*}
  \lim_{\substack{n,M\rightarrow \infty\\n/M\to\lambda}} A(k;M,n) = 
      \frac{\lambda^k}{k!}e^{-\lambda}
      \quad\text{and}\quad 
  \lim_{\substack{n,M\rightarrow \infty\\n/M\to\lambda}} M_2(k;M,n) = 0\,.
\end{equation*}
\end{theorem}

\section{Sharp uniform estimate of the average and of the second moment}\label{SectionUniformLimit}

\begin{theorem}\label{Theorem3}
Let $\lambda\in(0,1]$, $\gamma\in[0,1)$ and let $k\ge 0$  be integer.
Suppose the integer variables $M$ and $n$  satisfy the inequalities $k\le n\le M$  and 
$M=n/\lambda +O(n^\gamma)$, uniformly on $\lambda$ and $k$ as $n$ tends to infinity while $\gamma$ 
is fixed. 
Then
\begin{align}\label{}
   A(k;M,n) &=
\frac{\lambda^k}{k!}e^{-\lambda}\left(1+O\big(\lambda(\lambda+k)n^{\gamma-1}\big)\right)
\label{eqAu}
\intertext{and}
  M_2(k;M,n) &=O\left(\frac{\lambda^k}{k!}\cdot\frac{1}{n}
	\Big(1+\frac{\lambda^{k+1}(\lambda+k)}{k!}n^{\gamma}\Big)\right),
\label{eqM2u}
\end{align}
where the constants implied in the estimates are independent of $\lambda, \gamma$ and $k$.
\end{theorem}

\begin{proof}
 By the hypothesis of the theorem, we see that 
 $n/M = 1/(\lambda^{-1}+O(n^{\gamma-1}))$, so that 
 $1/M = \lambda (n^{-1}+O(\lambda n^{\gamma-2}))$ and
 $n/M = \lambda (1+O(\lambda n^{\gamma-1}))$.
 Then the asymptotic approximations of the main exp-log functions involved are:
 \begin{equation}\label{eqpow}
   \begin{split}
    \left(\frac{n}{M}\right)^{k}
      = \lambda^k\Big(1+O\big(\lambda  n^{\gamma-1}\big)\Big)^k
      = \lambda^k\Big(1+O\big(\lambda k n^{\gamma-1}\big)\Big)
   \end{split}
\end{equation}
and, if  $a=1$ or $2$, then 
 \begin{equation}\label{eqexp}
   \begin{split}
\left(1-\frac{a}{M}\right)^{n-k}=e^{(n-k)\log(1-a/M)}
  = e^{\frac{-an}{M}+O\left(\frac{k}{M}\right)}
  &= e^{-a\lambda}\left(1+O\big(\lambda^2n^{\gamma-1}+\lambda k n^{-1}\big)\right).
   \end{split}
\end{equation}
 On combining \eqref{eqpk}, \eqref{eqpow} and \eqref{eqexp}, we find that
 \begin{equation*}
   \begin{split}
   A(k;M,n) &= \frac{1}{k!}\cdot \lambda^k\Big(1+O\big(\lambda k n^{\gamma-1}\big)\Big)
      e^{-\lambda}\left(1+O\big(\lambda^2n^{\gamma-1}+\lambda k n^{-1}\big)\right)\\
       &= 
\frac{\lambda^k}{k!}e^{-\lambda}\left(1+O\big(\lambda(\lambda+k)n^{\gamma-1}\big)\right)\,.
   \end{split}
\end{equation*}
 The estimation of $M_2(k;M,n)$ is obtained similarly on combining 
 \eqref{eqM2S2}, \eqref{eqS2bis}, \eqref{eqSISII},\eqref{eqpow} and \eqref{eqexp}, which  completes 
the proof of 
the theorem,
 
\end{proof}

\smallskip
Now we can prove the uniform result in Theorem~\ref{TheoremMain}.
%
 For any $\eta >0$, let us split $\cT$ in two disjoint parts, that is,
 $\cT=\cT^{<}(\eta)\cup \cT^{\ge}(\eta)$ and $\cT^{<}(\eta)\cap \cT^{\ge}(\eta)=\emptyset$, where
 \begin{equation*}
   \begin{split}
   \cT^{<}(\eta)  = \Big\{\bx\in\cT :\ \Big|\frac {m_k(\bx)}{n}
     -\frac{\lambda^k}{k!}e^{-\lambda}\Big|<\eta\Big\}\quad  \text{ and }  \quad
   \cT^{\ge}(\eta)  = \Big\{\bx\in\cT :\ \Big|\frac {m_k(\bx)}{n}
     -\frac{\lambda^k}{k!}e^{-\lambda}\Big|\ge\eta\Big\}\,.\     
   \end{split}
\end{equation*}
Then, accordingly,
\begin{equation*}
   \begin{split}
   M_2(k;M,n)&=\frac{1}{|\cT|} \sum_{\bx \in \cT^{<}(\eta)}		
	       \left(\frac{m_k(\bx)}{n}-\frac{\lambda^k}{k!}e^{-\lambda}\right)^2
	       +\frac{1}{|\cT|} \sum_{\bx \in \cT^{\ge}(\eta)}		
	       \left(\frac{m_k(\bx)}{n}-\frac{\lambda^k}{k!}e^{-\lambda}\right)^2.\\
   \end{split}
\end{equation*}
Ignoring the contribution of the first sum, we obtain for $M_2(k;M,n)$ a bound from below:
\begin{equation*}
   \begin{split}
   M_2(k;M,n)\ge  \frac{1}{|\cT|} \sum_{\bx \in \cT^{\ge}(\eta)}\eta^2	
   =\frac{\eta^2}{|\cT|}\big(|\cT|-|\cT^{<}(\eta)|\big)\,.
   \end{split}
\end{equation*}
This can be rewritten conveniently as the following inequality:
\begin{equation}\label{eqM2last}
   \begin{split}
  \frac{|\cT^{<}(\eta)|}{|\cT|}\ge 1-\eta^{-2}M_2(k;M,n)\,.
   \end{split}
\end{equation}
On using the estimate~\eqref{eqM2u} on the right side of \eqref{eqM2last} and choosing 
$\eta = n^{-\delta}$, we find that
\begin{equation*}
   \begin{split}
   \frac{|\cT^{<}(n^{-\delta})|}{|\cT|}\ge 
   1-O\left(\frac{\lambda^k}{k!}\cdot\frac{1}{n^{1-2\delta}}
	\Big(1+\frac{\lambda^{k+1}(\lambda+k)}{k!}n^{\gamma}\Big)\right),
   \end{split}
\end{equation*}
which concludes the proof of the theorem.

\bigskip
\bigskip

\textbf{Acknowledgements: } Calculations and plots created using the free open-source mathematics
software system \texttt{SAGE}: 
\href{http://www.sagemath.org}{http://www.sagemath.org}.


\end{document}